\documentclass{amsart}
\usepackage[utf8]{inputenc}
\usepackage{amssymb,latexsym}
\usepackage{amsmath}
\usepackage{graphicx}
\usepackage{textcomp}
\usepackage[dvipsnames]{xcolor}

\usepackage{amsthm,amssymb,enumerate,graphicx, tikz}
\usepackage{amscd}
\usepackage{setspace}
\usepackage{comment}
\usepackage{hyperref}
\usepackage{cleveref}

\def\@logofont{\footnotesize}
\def\@setaddresses{\par
  \nobreak \begingroup
  \footnotesize
  \def\author##1{\nobreak\addvspace\bigskipamount}%
  \def\\{\par\nobreak}%
  \interlinepenalty\@M
  \def\address##1##2{\begingroup
    \par\addvspace\bigskipamount\indent
    \@ifnotempty{##1}{(\ignorespaces##1\unskip) }%
    {\scshape\ignorespaces##2}\par\endgroup}%
  \def\curraddr##1##2{\begingroup
    \@ifnotempty{##2}{\nobreak\indent\curraddrname
      \@ifnotempty{##1}{, \ignorespaces##1\unskip}\/:\space
      ##2\par}\endgroup}%
  \def\email##1##2{\begingroup
    \@ifnotempty{##2}{\nobreak\indent\emailaddrname
      \@ifnotempty{##1}{, \ignorespaces##1\unskip}\/:\space
      \ttfamily##2\par}\endgroup}%
  \def\urladdr##1##2{\begingroup
    \def~{\char`\~}%
    \@ifnotempty{##2}{\nobreak\indent\urladdrname
      \@ifnotempty{##1}{, \ignorespaces##1\unskip}\/:\space
      \ttfamily##2\par}\endgroup}%
  \addresses
  \endgroup
}
\renewcommand*\subjclass[2][2010]{%
  \def\@subjclass{#2}%
  \@ifundefined{subjclassname@#1}{%
    \ClassWarning{\@classname}{Unknown edition (#1) of Mathematics
      Subject Classification; using '2000'.}%
  }{%
    \@xp\let\@xp\subjclassname\csname subjclassname@#1\endcsname
  }%
}

\usepackage{graphicx,amsmath,amssymb}
\usepackage{algorithm}
\usepackage{mathdots, comment}
\usepackage[noend]{algpseudocode}

\newtheorem{theorem}{Theorem}[section]

\newtheorem*{theorem*}{Theorem}
\newtheorem{proposition}[theorem]{Proposition}
\newtheorem{lemma}[theorem]{Lemma}

\newtheorem{question}[theorem]{Question}

\newtheorem{corollary}[theorem]{Corollary}
\newtheorem{claim}[theorem]{Claim}

\theoremstyle{definition}
\newtheorem{definition}[theorem]{Definition}

\theoremstyle{remark}
\newtheorem{remark}[theorem]{Remark}
\newtheorem{example}[theorem]{Example}

\newcommand{\F}{\mathcal{F}}
\newcommand{\M}{\mathcal{M}}

\newcommand{\I}{\mathcal{I}}

\newcommand{\N}{\mathcal{N}}

\begin{document}
\title{Matchings in matroids over abelian groups}
\thanks{Dedicated to Professor Shmuel Friedland on his 80th birthday.}
\thanks{\textbf{Keywords and phrases}. matchable bases, sparse paving matroids, transversal matroids.}
\thanks{\textbf{2020 Mathematics Subject Classification}. Primary: 05B35 ; Secondary: 05D15, 05E16.}

\author[M. Aliabadi]{Mohsen Aliabadi$^{1,*}$ \and Shira Zerbib$^2$}
\thanks{$^1$Department of Mathematics, University of California, San Diego, 
9500 Gilman Dr, La Jolla, CA 92093, USA.  \url{maliabadisr@ucsd.edu}.\\
$^2$Department of Mathematics, Iowa State University,
411 Morrill Road
Ames, IA 50011, USA. \url{zerbib@iastate.edu}. Shira Zerbib was supported by NSF grant DMS-1953929.}
\thanks{$^*$ Corresponding Author.}


\begin{abstract}
 We formulate and prove matroid analogues of results concerning matchings in groups. A matching in an abelian group $(G,+)$ is a bijection $f:A\to B$ between two finite subsets $A,B$ of  $G$ satisfying $a+f(a)\notin A$ for all $a\in A$. A group $G$ has the matching property if for every two finite subsets  $A,B \subset G$  of the same size with $0 \notin B$, there exists a matching from $A$ to $B$. In  \cite{Losonczy} it was proved that an abelian group has the matching property if and only if it is torsion-free or cyclic of prime order. Here we consider a similar question in a matroid setting. We
introduce an analogous notion of matching between matroids whose ground sets are subsets of an abelian group $G$, and we obtain criteria for the existence of such matchings.  Our tools are classical theorems in matroid theory, group theory and additive number theory. 
\end{abstract}

\maketitle

\section{Introduction}

\subsection{Matchings in groups.}

The notion of matching in groups was  first introduced in \cite{Fan} by Fan and Losonczy, who used matchings in $\mathbb{Z}^n$ to study an old problem of Wakeford \cite{Wakeford} concerning canonical forms for symmetric tensors.

Let $(G,+)$ be an abelian  group with neutral element $0$.  
Let $A,B$ be finite subsets of the same cardinality of 
$G$, so that $0\notin B$. A {\it matching} from $A$ to $B$ is  a bijection $f:A\to B$ such that  $a+f(a)\not\in A$ for every $a\in A$. Clearly, the conditions  $|A|=|B|$ and $0 \notin B$ are necessary for the existence of such a bijection. A matching is called {\it symmetric} if $A=B$, and otherwise it is {\it asymmetric}.
If there exists a matching from $A$ to $B$, then $A$ is said to be {\it matched} or {\it matchable} to $B$. 
The  group $G$ is said to satisfy the  {\it matching property} if for every two finite subsets $A,B\subset G$ such that $|A|=|B|$ and $0 \notin B$, $A$ is matched to $B$.

A characterization of abelian groups satisfying the matching property, as well as a necessary and sufficient condition  for the existence of symmetric matchings,  were obtained by Losonczy in \cite{Losonczy}:

 \begin{theorem}\cite{Losonczy}\label{symmetric matching}
Let $G$ be an abelian group and let $A$ be a nonempty finite subset of $G$.
Then there is a matching from $A$ to itself if and only if $0 \notin A$.
\end{theorem}

 \begin{theorem}\cite{Losonczy}\label{matching property}
An abelian group $G$ satisfies the matching property if and
only if $G$ is torsion-free or cyclic of prime order.
\end{theorem}

These results were extended to arbitrary groups in \cite{Eliahou 1} and to linear subspaces of a field extension in \cite{Eliahou 2}. The linear version
was further studied  in \cite{Aliabadi 0, Aliabadi 1, Aliabadi 3}. See also \cite{Hamidoune} for enumerative aspect of matchings and \cite{Alon} for the study a certain subclass of matchings called {\em acyclic matchings}.

In this paper we introduce an analogous notion of matchings in matroids whose ground sets are subsets of an abelian group. We develop similar criteria for the existence of matching properties for such matroids. 


\subsection{Matching in matroids}\label{Matching in matroids}
Matroids were first defined independently in the 1930s by Nakasawa \cite{Nakasawa} and Whitney
\cite{Whitney}, to simultaneously
capture notions of independence in linear algebra and graph theory. 
The theory of matroids has deep connections with many other areas, including field theory, matching
theory, submodular optimization, Lie combinatorics, and total positivity (see recent survey in \cite{Ardila}).

  We choose one of many equivalent definitions.
A {\it matroid} $M$ is a pair $(E, \mathcal{I})$ where $E=E(M)$ is a finite  {\em ground set}  and $\mathcal{I}$ is a family of subsets of $E$, called {\em independent sets},  satisfying the following  conditions:
\begin{itemize}
\item $\emptyset \in \I$.
\item If $X\in \I$ and $Y\subset X$ then $Y\in \I$.
\item The {\em augmentation property}: If $X, Y \in \I$  and $|X|>|Y|$ then there exists $x\in X$ so that $Y\cup \{x\} \in \I$. 
\end{itemize}

Let $M=(E, \mathcal{I})$ be a matroid. 
The {{\em rank}} of a subset $X\subset E$ is given by
$$r_M(X)=
r(X)=max\{|X\cap I|: I\in \mathcal{I}\}.
$$

The rank of a matroid $M$, denoted by $r(M)$, is defined to be $r_M(E(M))$.
A set $X\subset E$ is called {{\em dependent}} if it is not independent.
A maximal independent set is called a {{\em basis}}. It follows from the augmentation property that every two bases have the same size.
Given a matroid $M=(E,\I)$, the {{\em dual matroid}} $M^*=(E, \mathcal{I}^*)$ is defined so that the
bases in $M^*$ are exactly the complements of the bases in $M$.

A matroid $M=(E,\I)$ of rank $n$ is said to be  a {{\em paving matroid}} if every $(n-1)$-subset of $E$ is independent.
If $M$ and $M^*$ are both paving matroids then $M$ is called a {\em sparse paving matroid}.\\

Throughout the paper  $(G,+)$ is assumed to be an abelian additive group with neutral element $0$, and let $p(G)$ denote the smallest cardinality of a non-zero subgroup of $G$. We may assume that all matroids are loopless (see Remark \ref{loop remark}.)
We say that $M=(E,\mathcal{I})$ is a matroid over $G$ if $E$ is a subset of $G$.  We now introduce the definition of {\em matchings in matroids over an abelian group}. Let $S_n$ denote the permutation group on $n$ elements. 

\begin{definition}\label{def, matching matroid}  Let  $(G,+)$ be an abelian group.
\begin{enumerate}
\item  Let $M$ and $N$ be two matroids over $G$ with $r(M)=r(N)=n>0$. Let $\mathcal{M}=\{a_1,\ldots,a_n\}$ and $\mathcal{N}=\{b_1,\ldots,b_n\}$ be ordered bases of $M$ and $N$, respectively. We say $\mathcal{M}$ is {{\em matched}} to $\mathcal{N}$ if $a_i+b_i\notin E(M)$, for all $1\leq i\leq n$. 
    
    \item We say that $M$ is {{\em matched}}  to $N$ if for every basis $\mathcal{M}$ of $M$ there exists a basis $\mathcal{N}$ of $N$, so that $\mathcal{M}$ is matched to $\mathcal{N}$. 
    \item The group $G$ has the {\em matroid matching property}  if for every two matroids $M$ and $N$ over $G$ with $r(M)=r(N)=n>0$ and $0\notin E(N)$, $M$ is matched to $N$.

 \end{enumerate}
\end{definition}

\begin{remark}
Note that 
if $\mathcal{M}$ is matched $\mathcal{N}$, then
$a_i+b_i\notin E(M),$
and in particular, $a_i+b_i\notin \mathcal{M}$, for all $1\leq i \leq n$. It follows that the map $a_i\mapsto b_i$ is a matching in the group sense between the  subsets $\mathcal{M}$ and $\mathcal{N}$ of $G$. In this sense, our definition of matchings in matroids is compatible with the definition of matchings in groups. 

Furthermore, our definition is compatible with the definition of matchability in vector spaces over field extensions, as defined by Eliahou and Lecouvey in \cite{Eliahou 2}.  Let $K\subset F$ be fields, $A,B\subset F$  $n$-dimensional $K$-subspaces of $F$, and $\mathcal{A}=\{a_1,\ldots,a_n\}$, $\mathcal{B}=\{b_1,\ldots,b_n\}$  ordered bases of $A,B$ (as linear vector spaces), respectively. 
Then $\mathcal{A}$ is {\it matched} to $\mathcal{B}$ if 
\begin{equation}\label{eqrem}
    a^{-1}_iA\cap B\subset \langle b_1,\ldots,\hat{b}_i,\ldots,b_n\rangle,
\end{equation}
for every $1\leq i\leq n$, where $\langle b_1,\ldots,\hat{b}_i,\ldots,b_n\rangle$ is the vector space spanned  by  $\mathcal{B}\setminus\{b_i\}$. The vector space  $A$ is {\it matched}  to the vector space $B$ if every basis of $A$ is matched to a basis of $B$.

Let $G$ be the multiplicative group  $F\setminus \{0\}$, and define the matroids over $G$ $M=(A\setminus \{0\},\mathcal{I})$, $N=(B\setminus \{0\},\mathcal{I}')$, where $\mathcal{I},\I'$ are the collections of linearly independent subsets of $A,B$, respectively. Then 
\[r(M)=\dim_K(A)=n=\dim_K(B)=r(N).\] Suppose $\mathcal{A}$  is matched to $\mathcal{B}$ as linear vector spaces. Then it follows from (\ref{eqrem}) that $a_ib_i\notin A$, and therefore   $\mathcal{A}$ is matched to  $\mathcal{B}$ also in the matroid sense. 
\end{remark}

\medskip

The purpose of this paper is to investigate matroid analogs of Theorems \ref{symmetric matching} and \ref{matching property}.
 We study the following questions: 

\begin{question}\label{matroid automatch}
Let $M=(E,\mathcal{I}$) be a matroid over $G$. Is it true that $M$ is matched to itself if and only if $0 \notin E$?
\end{question}

\begin{question}\label{matroid match prop} 
Is it true that $G$ has the matroid matching property if and only if it is  torsion-free or  cyclic of prime order?
\end{question}

We will show that the ``only if" part is true in both questions. 

\begin{proposition}\label{prop:onlyif}
 We have the following:
\begin{enumerate}
   \item Let $M$ be a matroid over $G$.  If  $M$ is matched to itself then  $0\notin E(M)$.
\item If $G$ satisfies the matroid matching property then $G$ is either torsion-free or cyclic of prime order.
\end{enumerate}
\end{proposition}

The ``if" part is not true  in general, as is shown in Section \ref{sec:examples}. However, if one restricts the discussion to some special (large) classes of  matroids, then the answer becomes positive in both cases. Our main results give a partial characterization of the cases at which the answer is positive.

 The paper is organized as follows. 
In Section \ref{sec:ourmainresults} we state our main results. 
In Section \ref{sec:prelim} we state some known results and prove a few lemmas in matroid theory and group theory that will be used in our proofs. 
In Section \ref{sec:examples} we give two simple  examples showing that the ``if" direction in Questions \ref{matroid automatch} and \ref{matroid match prop} is not true in general.
 Sections \ref{Symmetric} and \ref{Asymmetric} contain the proofs of our main results concerning 
symmetric and asymmetric matchings, respectively. Finally, in Section 6 we give a concluding remark.

\section{Main results}\label{sec:ourmainresults}
We next state our main theorems. Needed definitions from matroid theory  appear in Section \ref{sec:prelim}.

A matroid $M=(E,\I)$ of rank $n$ is said to be a {{\em paving matroid}} if all its circuits are of size at least $n$, that is, if every $(n-1)$-subset of $E$ is independent.
If $M$ and $M^*$ are both paving matroids then $M$ is called a {\em sparse paving matroid}. 
We start by showing that the answer to Question \ref{matroid automatch} is positive whenever $M$ is a sparse paving matroid.

\begin{theorem}\label{sparse sym main}
 Let $M$ be a sparse paving matroid over $G$. If
$0\notin E(M)$ then
$M$ is matched to itself.
\end{theorem}

The rest of our results concern asymmetric matchings. We begin this discussion by investigating the following question:

\begin{question}\label{main sparse asy}
Let $G$ be an abelian torsion-free group or a cyclic group of prime order. Let $M$ be a matroid over $G$ and let $N$ be a sparse paving matroid over $G$, both of the same rank $n$, so that $|E(M)|\leq |E(N)|$ and $0\notin E(N)$. Is $M$ matched to $N$?
\end{question}
\begin{remark}
Notice that in Question \ref{main sparse asy}, the condition $|E(M)|\leq |E(N)|$ is imposed as otherwise if $|E(M)|$ is arbitrarily larger than $|E(N)|$, then the matchability condition $a_i+b_i\notin E(M)$ may not hold in most situations. 
\end{remark}
Note that without imposing the condition of $N$ being sparse paving, one cannot guarantee that $M$ is matched to $N$, as is shown in Section \ref{sec:examples} (see Example \ref{counterexample (asymmetric)}). 

In the next theorem  we provide four different conditions for $M$ to be matched to $N$.
We need the following definitions. 
Let $G$ be an abelian group, $x, a\in G$ and let $k$ be a positive integer.  A {\em progression} of length $k$ with difference $x$ and initial term $a$ is a subset of $G$ of the form
$
\{a,a+x,a+2x,...,a+(k-1)x\}.
$.  We say $A$ is a {\em{semi-progression}} if $A\setminus\{a\}$ is a progression, for some $a\in A$.

\begin{theorem}\label{Asy sparse paving}
 Let $M$ be a matroid over $G$ and let $N$ be a sparse paving matroid over $G$, both of the same rank $n$, so that $0\notin E(N)$. Assume further that one of the following conditions holds:
 \begin{enumerate}
     \item $|E(M)| <\min\{|E(N)|-1, p(G)\}$, or
     \item $E(M)$ is not a progression, $G$ is finite, $|E(M)|=|E(N)|-1$ and $|E(N)|<p(G)$, or
     \item $E(M)$ is neither a progression nor a semi-progression, $G$ is finite and $|E(M)|=|E(N)|< p(G)$, or
     \item $|E(M)|<|E(N)|-n-1$.
    
 \end{enumerate}
 Then $M$ is matched to $N$
\end{theorem}

Conditions (1) implies:
\begin{corollary}\label{main corollary 1}
Suppose  $G$ is an abelian torsion-free group or a cyclic group of prime order. Let $M$ be a matroid over $G$ and $N$ be a sparse paving matroid over $G$, both of the same rank $n$, so that
$|E(M)|< |E(N)|-1$. If $0\notin E(N)$, then $M$ is matched to $N$.
\end{corollary}
Conditions (2) and (3) imply:
\begin{corollary}\label{main corollary 2}
Let $p$ be a prime number. Let $M$ be a matroid over $\mathbb{Z}/p\mathbb{Z}$ and $N$ be a sparse paving matroid over $\mathbb{Z}/p\mathbb{Z}$, both of the same rank $n$, so that
\begin{enumerate}

     \item $E(M)$ is not a progression, and $|E(M)|=||E(N)|-1$, or
    \item $E(M)$ is neither a progression nor a semi-progression,  and $|E(M)|=|E(N)|$.
    \end{enumerate}
    If $0\notin E(N)$, then $M$ is matched to $N$.
\end{corollary}

It follows from Corollary \ref{main corollary 1} and Corollary  \ref{main corollary 2} that Question \ref{main sparse asy} boils down to the following questions in which the remaining cases of matchability of matroids are addressed.

\begin{question}\label{remaining cases infinite}
Let $G$ be an abelian torsion-free group. Let $M$ be a matroid over $G$ and $N$ be a sparse paving matroid over $G$, both of the same rank $n$, so that $|E(M)|=|E(N)|-1$ or $|E(M)|=|E(N)|$. If $0\notin E(N)$, is $M$ matched to $N$?
\end{question}

\begin{question}\label{remaining cases finite}
Let $p$ be a prime number. Let $M$ be a matroid over $\mathbb{Z}/p\mathbb{Z}$ and $N$ be a sparse paving matroid over $\mathbb{Z}/p\mathbb{Z}$, both of the same rank $n$, so that 
\begin{enumerate}
    \item $|E(M)|=|E(N)|-1$ and $E(M)$ is a progression, or
    \item $|E(M)|=|E(N)|$ and $E(M)$ is a progression or $E(M)$ is a semi-progression.
\end{enumerate}
If $0\notin E(N)$, is $M$ matched to $N$?
\end{question}

In the special case where $N$ is a uniform  matroid, we can prove matchability for the general case when $|E(M)|\leq|E(N)|$.

\begin{proposition}\label{asymmetric, uniform}
 Let $M$ be any matroid over $G$ and let $N$ be a uniform matroid over $G$, both of the same rank $n$, such that $0\notin E(N)$. If 
$|E(M)|\leq|E(N)|<p(G)$, then $M$ is matched to $N$.
\end{proposition}

One of the cases that is not covered by Theorem \ref{Asy sparse paving} is the following:
\begin{question}\label{almost complete matroids}
Let $G$ be an abelian group and let $M,N$ be matroids over $G$ of rank $n$, so that  $N$ is sparse paving,  $0\notin E(N)$ and $|E(M)|=|E(N)|=n+1<p(G)$.  Is it true that  $M$ is matched to $N$?
\end{question}

Though we have not been able to settle this case, we obtain a partial result whose proof relies on the existence of a total order on the set $$E(M)\cup E(N)\cup (E(M)+E(N))\cup\{0\}$$  that is compatible with the group structure of $G$.
Given an abelian group $G$ and a subset $A\subset G$, a total order on $A$ is {\em compatible with  the group structure}, if for every $a, b, c \in A$, $a\preceq b$ implies $a+c \preceq b+c$. 
 Assuming that $\preceq$ is a total order on $A\cup\{0\}$, we say that an element $x \in A$ is {\it positive} if $0\prec x$. $x$ is {\it negative} if $x \prec 0$ (see Section \ref{sec:total order} for definitions of positive and negative subsets of an abelian group.)

On the existence of such orders, see Section \ref{sec:total order}. In particular, Theorem \ref{linearly ordered groups} and Lemma \ref{rectification} imply that such an order exists whenever $G$ is torsion-free, or when $$|E(M)\cup E(N)\cup (E(M)+E(N))\cup\{0\}| < \lceil{\log_2 (p(G))\rceil}.$$

Note that in the following theorem, we only assume that $N$ is a paving matroid but not necessarily a sparse paving matroid as before. 

\begin{theorem}\label{asymmetric, sparse paving}
Let  $M$ be a matroid over $G$ and let  $N$ be a paving matroid over $G$, both of rank $n$, so that $|E(M)|=|E(N)|=n+1< p(G)$. Assume  that there exists a total order $\preceq$ on $E(M)\cup E(N) \cup (E(M)+E(N))\cup\{0\}$ that is compatible with the group structure of $G$, so that $E(M)$ and $E(N)$ are both positive, and  $\max(E(M))\notin E(M)+E(N)$. 
Then $M$ is matched to $N$.
\end{theorem}

By symmetry,  Theorem \ref{asymmetric, sparse paving} also holds when  $E(M)$ and $E(N)$ are both negative and $\min(E(M))\notin E(M)+E(N)$. 
This raises a natural question:
\begin{question}
Does Theorem \ref{asymmetric, sparse paving} hold if $E(M)$ and $E(N)$ contain positive and negative elements simultaneously?
\end{question}

In the following theorem, we prove an analogue of Theorem \ref{Asy sparse paving} (3) by removing the condition of $N$ being sparse paving, in the price of imposing the extra assumptions that $|(-a+E(M))\cap E(N)|\neq n$, for all $a\in E(M)$, and $|E(M)|=|E(N)|=n+1$.

\begin{theorem}\label{$|E(M)|=|E(N)|=n+1$}
Let $M$ and $N$ be  matroids over a finite abelian group $G$, both  of  rank $n$, so that $|E(M)|=|E(N)|=n+1<p(G)$. Assume $|(-a+E(M))\cap E(N)|\neq n$, for all $a\in E(M)$. Assume further that $E(M)$ is neither a progression nor a semi-progression. 
If $0\notin E(N)$ then $M$ is matched to $N$.
\end{theorem}
\begin{remark}
It is worth pointing out that a condition akin to $|(-a+E(M))\cap E(N)|\neq n$ in Theorem \ref{$|E(M)|=|E(N)|=n+1$} appears in \cite[Theorem 1.3]{Hamidoune}. That is, it is proved that, in general, in the group setting for a subset $A\subset G$ and a Chowla subset $B\subset G$ both of the same size $n$, one can produce at least $\min\{\frac{n+1}{3}, \frac{|<B>|-n-1}{2|<B>|-n-4}\}$ matchings, (where $<B>$ stands for the subgroup generated by $B$), unless certain situations, one of which is when $|(-a+A)\cap B|=n-1$, for some $a\in A$. 
 
\end{remark}
 The special case where $|E(M)|=|E(N)|=n+1$ may be treated in a different way from Theorem \ref{$|E(M)|=|E(N)|=n+1$}, without any additive assumption on the ground sets $E(M)$ and $E(N)$, but  by imposing a matroidal condition on $N$ as follows:
 
 \begin{theorem}\label{coloopless sparse paving}
 Let $M$ and $N$ be two matroids over an abelian group $G$, both of rank $n$ so that $|E(M)|=|E(N)|=n+1<p(G)$. Assume further that $N$ is a coloopless sparse paving matroid. If $0\notin E(N)$ then $M$ is matched to $N$.
 \end{theorem}

We now turn our attention to transversal matroids. Let a set $E$ be partitioned into disjoint sets $E_1,...,E_l$ and let 
$$\mathcal{I}=\{X\subset E:|X\cap E_i|\leq 1, \text{ for all}\ i=1,...,l\}.$$ Then $M=(E, \mathcal{I})$ is a matroid called a  {\em transversal matroid}. 

 Similar to  Theorem \ref{asymmetric, sparse paving}, we assume the existence of a total order on the set $E(M)\cup E(N)\cup (E(M)+E(N))\cup\{0\}$ that is compatible with the group structure of $G$.

\begin{theorem}\label{partition, asy 1}
Let  $M=(E,\mathcal{I})$ and $N=(E',\mathcal{I}')$ be two transversal matroids over $G$, with partition sets $\{E_i\}_{i\in [n]}, \{E'_i\}_{i\in [n]}$, respectively, such that $|E_i|=|E'_i|$ for all $i$. Assume that there exists a total order $\preceq$ on $E\cup E'\cup (E+E')\cup\{0\}$  that is compatible with the group structure of $G$, such that the following conditions hold:
\begin{enumerate}
\item $E$ and $E'$ are both positive.
\item For all $1\leq i<j\leq n$, we have $E_i\prec  E_j$,
 $E'_i\prec  E'_j$,
and $|E_i|>|E_j|$.
\item $\max E \preceq \max E'$.
\end{enumerate}
Then $M$ is matched to $N$.
\end{theorem}

By a similar proof,  Theorem \ref{partition, asy 1}  also holds if we replace conditions (1)-(3) by their negative counterparts:
\begin{enumerate}
\item
$E$ and $E'$ are both negative.
\item
For all $1\leq i<j\leq n$, we have 
 $E_i\prec  E_j$,
$E'_i\prec  E'_j$,
and $|E_i|<|E_j|$.
\item
$\min E' \preceq \min E$.
\end{enumerate}
\medskip

We  use Theorem \ref{partition, asy 1} to prove:

\begin{theorem}\label{asymmetric, transversal matroid}
  Let $M=(E, \mathcal{I})$ and $N=(E', \mathcal{I'})$ be two transversal matroids over $G$, with partition sets $\{E_i\}_{i\in [n]}, \{E'_i\}_{i\in [n]}$, respectively, so that $|E_i|=|E'_i|$ for all $i$. Assume further that there exists a total order $\preceq$ on $E\cup E'\cup (E+E')\cup\{0\}$  that is compatible with the group structure of $G$, so that the following properties hold:
\begin{enumerate}
\item
There exists $k\in [n]$ 
such that 
\begin{enumerate}[(a)]
    \item $E_i, E_i'$ are  negative when  $i<k$ and  positive otherwise. 
    \item $E_k=-E_k'$. That is, $E_k=\{-a: a\in E_k'\}$.
    \item $|E_i|>|E_j|$ whenever $k<i<j$ or  $ j<i<k$. 
\end{enumerate}
\item
For all $i<j$, $E_i\prec E_j$ and $E'_i\prec E'_j$.
\item $\max E \preceq \max E'$, and $\min E' \preceq \min E$.
\end{enumerate}
Then $M$ is matched to $N$.
Furthermore, if  $k=1$  then we can remove the condition $\min E' \preceq \min E$, and if $k=n$ then we can remove the condition $\max E \preceq \max E'$. 
\end{theorem}

Theorems \ref{symmetric matching} and \ref{matching property}  were proved using Hall's marriage theorem \cite{Hall} together with a result due to Kneser (see \cite[Theorem 4.3]{Nathanson}) known as Kneser's additive theorem. A matroidal generalization of Hall's theorem was proved by Rado in \cite{Rado}, where a necessary and sufficient condition for the existence of independent transversals in matroids is given. Rado's theorem plays a  role in our proofs. Various matroidal analogues of Kneser's additive theorem appear in the literature \cite{Schrijver, DeVos}, but none of them are applicable in our setting. We fill this gap to some extent by using the group version of Kneser's theorem as well as results in group theory that enable us to equip a group  $G$  with a total order that preserves the group's additive properties. 
In the cases where such order does not exist we use a rectification principle due to Lev \cite{Lev}.
A corollary due to Eliahou and Lecouvey \cite{Eliahou 1} of Kemperman's additive theorem \cite{Kemperman}, is also invoked in our proofs.

\section{Preliminaries}\label{sec:prelim}

\subsection{Notions in matroid theory} In this section we give further definitions and state a few results from matroid theory. For a wider introduction, we refer the reader to \cite{Oxley}.

The rank function defined in Section \ref{Matching in matroids} satisfies certain nice properties. 
\begin{proposition}\label{rank func}
Let $M=(E,\mathcal{I})$ be a matroid, and $X,Y\subset E$. Then $|X|\geq r(X)$ and $r(X)+r(Y) \geq r(X\cup Y) + r(X\cap Y).$ 
\end{proposition}
The  inequality in the proposition is frequently referred to as   {\em the submodular
inequality for the rank}.
Using it we can bound the rank of a complement $X^c=E\setminus X$ of a set $X\subset E$.

\begin{proposition}\label{submodularity}
Assume that $r(M)=n$, and $r(X)\leq m$, for some $1\leq m \leq n$. Then $r(X^c)\geq n-m$.
\end{proposition}
\begin{proof} We have
 $       n=r(M)=r(X\cup X^c)\leq r(X)+r(X^c)-r(X\cap X^c)
        =r(X)+r(X^c)
\leq m+r(X^c).$
\end{proof}

A {{\em circuit}} in a matroid $M$ is a minimal dependent set. That is, a dependent set whose proper subsets are all independent.
A {{\em flat}} in  $M$ is a set $F \subset E(M)$ with the property that adjoining any new element to $F$ strictly increases its rank. A flat of rank $r(M)-1$ is called a {{\em hyperplane}}. A set $H\subset E(M)$ is called a {{\em circuit-hyperplane}} of $M$ if it is at the same time a circuit and a hyperplane. Note that if $M$ is a matroid of rank $n$,  $H$  is a circuit-hyperplane of $M$ and $x\in E(M)\setminus H$, then $r(H\cup \{x\})=n$.\\

An element that forms a single-element circuit of $M$ is called a {\it{loop}}. Equivalently, an element is a loop if it belongs to no basis \cite{Oxley}. 
An element that belongs to no circuit is called a {\it{coloop}}. Equivalently, an element is a coloop if it belongs to every basis \cite{Oxley}. We say that $M$ is  {\it{loopless}} ({\it{coloopless}}) if it does not have a loop (coloop).
A matroid is called {\it{loopless-coloopless}} if it is loopless and coloopless. 
In this
paper, we will assume that all matroids are loopless; this is not a restrictive
assumption because in general one can just remove from $E$ elements that do
not belong to any independent set.

Back to paving matroids, recall that a matroid $M=(E,\I)$ of rank $n$ is said to be  a  paving matroid if all its circuits are of size at least $n$, that is, if every $(n-1)$-subset of $E$ is independent. We mentioned before that sparse paving matroids are a specific subfamily of paving matroids.

An alternative characterization of sparse paving matroids was proved in  \cite{Bonin}.
\begin{theorem}\cite{Bonin}\label{sparse paving alternative}
A matroid $M$ of rank $n$ is a sparse paving matroid if and only if every $n$-subset of $E(M)$ is either a basis or a circuit-hyperplane.
\end{theorem}
In the following theorem \cite{Ferroni}, an upper bound for the number of circuit-hyperplanes of a sparse paving matroid is given.
\begin{theorem}\label{counting hyperplanes}
Let $M$ be a sparse paving matroid of rank $n$ having $m$ elements.
Then, the number of circuit-hyperplanes $\lambda$ of $M$ satisfies:
\begin{equation*}
    \lambda\leq {m\choose n}\min \{\frac{1}{n+1}, \frac{1}{m-n+1}\}.
\end{equation*}
\end{theorem}
It is conjectured in \cite{Crapo} that asymptotically, almost all matroids are sparse paving. That is, sparse paving
matroids predominate in any asymptotic enumeration of matroids. In pursuit of this conjecture, Pendavingh and van der Pol \cite{Pendavingh} showed that logarithmically almost all matroids are sparse paving matroids. This verifies that although sparse paving matroids seem to be very restrictive axiomatically, they form a  large class of matroids.




A $d$-partition of a set $E$ is a collection $\mathcal{S}$ of subsets of $E$, all of which of size at least $d$, so that the intersection of any two sets in $\mathcal{S}$  is of size at most $d-1$. The set $\mathcal{S} = \{E\}$ is a $d$-partition for any $d$, which we call a trivial
$d$-partition.
The connection between paving matroids and d-partitions is given by the following result whose proof can be found in \cite[Proposition 2.1.24]{Oxley}.

\begin{theorem} \label{(n-1) partitions}
If $M$ is a paving matroid of rank $n\geq 2$, then its hyperplanes form
a non-trivial $(n-1)$-partition of $E(M)$. That is, the intersection of any two hyperplanes in $M$ is of size at most $n-2$.
\end{theorem}

A matroid $M$ is called {\em free} if $r(M)=|E(M)|$. Note that a free matroid has only one basis $E(M)$. 

For integers $0\leq n \leq m$, let $E$ be an $m$-element set and let $\mathcal{I}$ be the collection of subsets of $E$ with at most $n$ elements. Then $M=(E, \mathcal{I}) $ is a matroid with rank $n$, called the {\em uniform matroid} and denoted by  $U_{n,m}$. In $U_{n,m}$ every set of size $n$ is a basis. Note that uniform matroids are sparse paving.

Let a set $E$ be partitioned into disjoint sets $E_1,...,E_l$ and let 
$$\mathcal{I}=\{X\subset E:|X\cap E_i|\leq k_i, \text{ for all}\ i=1,...,l\},$$
where $k_1,...,k_l$ are given non-negative integers, called {{\em partition parameters}}. Then $M=(E, \mathcal{I})$ is a matroid called a {{\em partition matroid}}. Note that the rank function of $M$ satisfies 
$r(X)=\sum_{i=1}^\ell \min(|E_i\cap X|,k_i)$
for all $X\subset E$. A partition matroid where $k_i=1$ for all $i$ is called a {\em transversal matroid}.

\subsection{Rado's theorem} Let $E$ be a finite set and let $\mathcal{F}=(F_1,\ldots,F_n)$ be a family of subsets of $E$. A {\em transversal}  of $\F$ is a tuple of distinct elements $(x_1,\ldots,x_n)$  satisfying $x_i\in F_i$, for all $i=1,\ldots,n$. 
 Let $M$ be a matroid over the ground set $E$ with rank function $r$. We say that  $\mathcal{F}$ admits an {\em independent transversal} if there exists a transversal $(x_1,\ldots,x_n)$ of $\F$, so that  $\{x_1,\dots,x_n\}$ is independent in $M$.

The following theorem of Rado \cite{Rado} gives a necessary and sufficient condition for the existence of independent transversals in $\F$.

\begin{theorem}\cite{Rado}\label{transversal, Rado}
Let $E$ be a finite set and $\mathcal{F}=(F_1,\ldots,F_n)$ be a family of subsets of $E$. Let $M$ be a matroid over ground set $E$ with rank function $r$. Then  $\mathcal{F}$ admits an independent transversal  if and only if for every $J\subset [n]$,  
$$r\left(\bigcup_{i\in J}F_i\right)\geq|J|.
$$

\end{theorem}
\subsection{Rank criteria for matchable bases}

 In the following proposition, we use Rado's theorem to
give criteria for the existence of matchings between two matroids in terms of the rank of a suitable set. 

\begin{proposition}\label{rank criteria} 
Let $G$ be an abelian group and $M$, $N$ be two matroids over $G$ of the same rank $n$.
Let $\mathcal{M}=\{a_1,...,a_n\}$ be a basis for $M$. Assume that for every $J\subset [n]$,   $$r_N\left(\bigcap_{i\in J}\left((-a_i+E(M))\cap E(N)\right)\right)\leq n-|J|.
$$ Then $\mathcal{M}$ is matched to some basis of $N$.
\end{proposition}
\begin{proof}
For $J\subset [n]$, set $N_J=\bigcap_{i\in J}\left((-a_i+E(M))\cap E(N)\right).$ Combining  the conditions of the proposition with Proposition \ref{submodularity}, we obtain 
$r_N\left(N_J^c\right)\geq |J|,$  where  complements  are  taken in $E(N)$.  For $i\in [n]$, define 
$$F_i=\left((-a_i+E(M))\cap E(N)\right)^c.$$
Observe that 
$\bigcup_{i\in J} F_i = N_J^c$, and therefore we have
$r_N(\bigcup_{i\in J} F_i)\geq |J|.$ 

Thus, by Theorem \ref{transversal, Rado},   the family  $(F_1,\ldots,F_n)$ has a transversal $(b_1,\ldots,b_n)$ that is independent in $N$. Since $r(N)=n$,  $\mathcal{N}=\{b_1,\ldots,b_n\}$ is a basis for $N$. Furthermore, $b_i\in\left((-a_i+E(M))\cap E(N)\right)^c$, entailing $a_i+b_i\notin E(M)$. Thus $\mathcal{M}$ is matched to $\mathcal{N}$.
\end{proof}

\subsection{Total orders in abelian groups}\label{sec:total order} 
The existence of a total order that is compatible with the group operation will be an essential building block in our proofs.

Given an abelian group $G$ and a subset $A\subset G$, a total order on $A$ is {\em compatible with the group structure}, if for every $a, b, c \in A$, $a\preceq b$ implies $a+c \preceq b+c$. Sometimes we use the notation $a \prec b$ if $a \preceq b$ and $a\neq b$.\\

 Let $A\subset G$. Assuming that $\preceq$ is a total order on $A\cup\{0\}$:
 \begin{itemize}
     \item An element $x \in A$ is {\it positive} if $0\prec x$. $x$ is {\it negative} if $x \prec 0$.
     \item $A$ is called {\em positive (negative)} if every  $x \in A$ is positive (negative). In this case, we write $0 \prec A$ ($A \prec 0$).
     
 \end{itemize}
      For subsets $A,B$ of $G$, we say that $B$ is greater than $A$ ($A\preceq B$) if $x\preceq y$ for any $x\in A$ and $y\in B$. If $B=\{b\}$ and $A\preceq B$, we simply write $A\preceq b$. We say $B$ is strictly greater than $A$ ($A\prec B$) if $a\prec b$, for every $a\in a$ and $b\in B$. We say If $A$ is a finite subset of $G$, we say that $x\in A$ is the {{\em maximum}} ({{\em minimum}}) element of $A$ if $a\preceq x$ ($x\preceq a$) for every $a\in A$, and we denote it by $\text{max}(A)$ ($\text{min}(A)$).

In \cite{Levi} Levi proved the following:
\begin{theorem}\cite{Levi} \label{linearly ordered groups}
An abelian group $G$ has a total order that is compatible with the group structure if and only if it is torsion-free.
\end{theorem}

When $G$  is not torsion-free, one may equip small enough subsets of $G$ with a total ordering that preserves certain additive properties of $G$. This is possible by utilizing the {\it rectifiability} in abelian groups, which we describe next.

Given an integer $t>0$, a map $\varphi:A\to S$ from a subset $A$ of a semigroup to a (potentially
different) semigroup $S$ is called a {\em Freiman isomorphism of order $t$} if the equality
$$
a_1+...+a_t=b_1+...+b_t
$$
holds if and only if 
$$
\varphi(a_1)+...+\varphi(a_t)=\varphi(b_1)+...+\varphi(b_t)
$$
holds, for every choice of $a_i, b_j \in A$, $1\leq i, j \leq t$. 
In this case we say that $A$ is {\em Freiman isomorphic of order $t$} to $\varphi(A)$.
We say that $A$ is {\em $t$-rectifiable} if it is Freiman isomorphic of order $t$ to a set of integers.
The following is proved in \cite{Lev}.
\begin{theorem}\cite{Lev}\label{$t$-rectifiable subsets}
Let $G$ be an abelian group with the non-trivial torsion subgroup, and let $t\geq 2$ be an integer. Let $p=p(G)$. Then any subset $A\subset G$ with $|A|\leq \lceil{\log_t p\rceil}$ is $t$-rectifiable, and the Freiman isomorphism induces a compatible order in $A$.
Moreover, there exists a non-$t$-rectifiable subset $A\subset G$ with $|A|=\lceil{\log_t p\rceil}+1$.
\end{theorem}

Using Theorem \ref{$t$-rectifiable subsets} we obtain the following lemma.
\begin{lemma}\label{rectification}
Let $G$ be an abelian group with the non-trivial torsion subgroup. Let $p=p(G)$ and $A$ be a subset of $G$ with $|A|< \lceil{\log_2 p\rceil}$. Then there exists a total ordering 
 $\preceq$ of $A$ that is compatible with  $G$. 
\end{lemma}
\begin{proof}
Invoking Theorem \ref{$t$-rectifiable subsets}, we conclude that there exists an injection $\varphi:A\cup{\{0\}}\to \mathbb Z$ such that for every $a, b, c, d \in A\cup{\{0\}}$, 
$$
a+b=c+d
$$
holds if and only if 
$$
\varphi(a)+\varphi(b)=\varphi(c)+\varphi(d).
$$
Replacing $\varphi$ with $\varphi-\varphi(0)$, we may assume that $\varphi(0)=0.$ Thus \begin{equation}\label{eq}
   \varphi(a)+ \varphi(b)=\varphi(a+b)+\varphi(0)=\varphi(a+b), 
\end{equation} where the first equality holds since $a+b=(a+b)+0$.

We now define a total ordering $\preceq$ on $A$ as follows: for every $a,b\in A$, $a\preceq b$ if and only if $\varphi(a) \leq \varphi(b)$. Then $\preceq$ is compatible with the group structure. Indeed, if $a, b, c$ are elements of $A$ with $a\preceq b$ then
$\varphi(a) \leq \varphi(b)$, implying that $\varphi(a)+ \varphi(c)\leq \varphi(b)+\varphi(c)$. Now by (\ref{eq}), $\varphi(a+c)\leq \varphi(b+c)$, implying $a+c\preceq b+c$, as desired.
\end{proof}
\subsection{Progressions}

Let $G$ be an abelian group, $x, a\in G$ and let $k$ be a positive integer.  A {\em progression} of length $k$ with difference $x$ and initial term $a$ is a subset of $G$ of the form
$
\{a,a+x,a+2x,...,a+(k-1)x\}.
$
A progression with difference $x$ is sometimes called an {\em{$x$-progression}}.  We say $A$ is a {\em{semi-progression}} if $A\setminus\{a\}$ is a progression, for some $a\in A$. A nonempty subset $A$ of $G$ is called a {\em Chowla subset} if
the order of every element of $A$ is greater than or equal to $|A|$+1.

We will need the following lemma for the proof of Theorem \ref{$|E(M)|=|E(N)|=n+1$}.

\begin{lemma}\label{progression}
Suppose $G$ is an abelian torsion-free group or a cyclic group
of prime order. Let $A=\{a_1,\dots,a_{n+1}\} \subsetneq G$ be a proper subset of $G$. Define $A_i=-a_i+A$, for any $i\in [n]$. If $A$ is not a progression, then $\bigcap_{i\in [n]}A_i=\{0\}$.  
\end{lemma}
\begin{proof}
Assume for contradiction that there exists $x\in \bigcap_{i\in [n]}A_i$ with $x\neq 0$. Define $D$ to be a directed graph whose vertices are $a_1,\dots, a_{n+1}$ with an edge from $a_i$ to $a_j$ if and only if $a_j-a_i=x$. Note that the indegree $\text{deg}^-(v)$  and outdegree $\text{deg}^+(v)$ of every vertex $v$  in $D$ is at most 1. 

We claim that for $i\in [n]$, $deg^+(a_i)=1$. Indeed,  $x\in A_i=-a_i+A$, and so there exists $a_j\in A$ such that $a_j-a_i=x$. Since $x\neq 0$ we have $j\neq i$, and thus $a_ia_j$ is an edge.  

Since the indegrees and outdegree are at most 1,  each connected component of $D$ is either a directed cycle or a directed path. 
Suppose there exists a connected component $C$ of $D$ that is a directed cycle $a_{i_1}a_{i_2}\dots a_{i_k}a_{i_1}$, for $k\in [n+1]$. Then $a_{i_j}-a_{i_{j-1}}=x$ for all $j\in [k]$ (where the addition in the indices is mod $k$),  implying 
$kx=0$. But this contradicts the fact that  $G$ is torsion-free or $G=\mathbb{Z}_p$ with $n+1<p$. 

We conclude that every connected component of $D$ is a directed path. Since every directed path has a vertex with a zero outdegree, and the outdegrees of all $a_1,\dots,a_n$ are not zero, $D$ must be a directed path  $a_{i_1}a_{i_2}\dots a_{i_n}a_{i_{n+1}}$ (ending at $a_{i_{n+1}}=a_{n+1}$). But this implies  $a_{i_j}=a_{i_1} + (j-1)x$ for all $j\in [n+1]$, showing that $A$ is a progression. This is a contradiction.
\end{proof}
\subsection{Abelian additive theory} 

We will find it necessary to appeal to a few notions from additive number theory, in particular, an addition
theorem of Kneser and a consequence of Kemperman's inequality. Given subsets $A$ and $B$ of $G$ and an integer $n>0$, define the {\em sumset} of $A$ and $B$ as
$$
A+B=\{a+b: a\in A, b\in B\},
$$
and let
$$
nA=\{a_1+...+a_n\mid  a_i\in A, \text{ for } 1\leq i\leq n \}.
$$
We will use  the following  theorem of Kneser, whose proof can be
found in \cite{Nathanson}.
\begin{theorem}[{Kneser's addition theorem} \cite{Nathanson}]\label{Kneser} 
Let $G$ be an abelian group and $A$ and $B$ are nonempty finite subsets of $G$. Then there exists a subgroup $H$ of $G$ such that
\begin{enumerate}
    \item $|A+B|\geq |A|+|B|-|H|$,
    \item $A+B+H=A+B.$
\end{enumerate}
\end{theorem}

The following classical theorem was proved by Kemperman \cite{Kemperman}.
\begin{lemma} \cite{Kemperman}\label{Kemperman's theorem}
Let $A$ and $B$ be nonempty finite subsets of a group $G$. Assume there exists an element $c\in A+B$ appearing exactly once as a sum
$c = a+b$ with $a\in A, b\in B$.   Then 
\begin{align*}
|A+B|\geq|A|+|B|-1.
\end{align*}
\end{lemma}
The following was essentially proved by  Eliahou-Lecouvey \cite{Eliahou 1} and it also follows from Lemma \ref{Kemperman's theorem}. We give the proof here for completion
\begin{lemma} \cite{Eliahou 1}\label{Kemperman's consequences}
Let $A$ and $B$ be nonempty finite subsets of a group $G$ such that $A$,  $B$ and $A+B$ are all contained in a subset $X$ of $G\setminus\{0\}$. Then 
\begin{align*}
|X|\geq|A|+|B|+1.
\end{align*}
\end{lemma}
\begin{proof}
Define $A_0=A\cup\{0\}$ and $B_0=B\cup\{0\}$. Then $0\in A_0+B_0$ and appears exactly once as a sum in $A_0+B_0$. That is, if $0=a+b$ with $a\in A$ and $b\in B$, since $0\notin A+B$, then $a=0$ or $b=0$, and hence $a=b=0$. Applying Theorem \ref{Kemperman's theorem} to $A_0$ and $B_0$ we have 
$$
|A_0+B_0|\geq |A_0|+|B_0|-1.
$$
Set $X=(A+B)\cup A\cup B$. Since $|A_0|=|A|+1$, $|B_0|=|B|+1$ and $A_0+B_0=(A+B)\cup A\cup B$, we have:
$$
|X|=|A_0+B_0|\geq |A|+|B|+1,
$$
as claimed.
\end{proof}
In the following lemma due to Kemperman, a certain family of critical pairs is characterized. Note that a pair $\{A,B\}$ of finite subsets of a group $G$ is called {\it{critical}} if $|G|>|A+B|=|A|+|B|-1$.
\begin{lemma} \cite{Kemperman 2}\label{Kemperman's critical theorem} Let $\{A,B\}$ be a critical pair of a finite abelian group $G$ with $|A|>1$, $|B|>1$ and $|A|+|B|-1\leq p(G)-2$. Then $A$ and $B$ are progressions with the same difference.
\end{lemma}


\section{Examples}\label{sec:examples}
In this section, we construct two examples showing that the ``if" part of Questions \ref{matroid automatch} and  \ref{matroid match prop} is not true. 

In the symmetric case, we have:

\begin{example}\label{counterexample(symmetric)}
For $n\ge 2$, let $M$ be a transversal matroid with  ground set $E=[2n] \subset\mathbb{Z}$, and with a partition $E_i=\{i\}$ for $1\leq i \leq n-1$ and $E_n=\{n, n+1, ..., 2n\}$. Then $r(M)=n$. Consider the basis $\mathcal{M}=[n]$ of $M$. Every basis $\mathcal{N}$ of $M$ is of the form $\mathcal{N}=[n-1]\cup \{m\}$, where $n\leq m \leq 2n$. We claim that $\M$ is not matched to $\N$. Indeed, if $m<2n$ then $1+b \in E(M)$ for every $b\in \N$, violating the matching condition. So assume $m=2n$ and $1\mapsto 2n$. Then $2+b \in E(M)$ for every $b\in \N\setminus \{2n\}$, violating the matching condition.    
Thus $M$ is not matched to itself although $0\notin E(M)$.
\end{example}

For asymmetric matchings consider the following examples.  

\begin{example}\label{counterexample (asymmetric)}
For $n\ge 2$, let $M$ be a uniform matroid of rank $n$ with ground set $E(M)=[2n]\subset\mathbb{Z}$. 
Let $N$ be the transversal matroid whose  ground set is $E(N)=[2n]\subset\mathbb{Z}$, with the partition $E_i=\{i\}$ for $1\leq i \leq n-1$ and $E_n=\{n,n+1,\dots, 2n\}$. Then $r_N(N)=n$. Consider the basis $\mathcal{M}=[n]$ of $M$. 
Then $\mathcal{M}$ cannot be matched to $\mathcal{N}$ of $N$ as $\mathcal{N}$ must be of the form $\mathcal{N}=\{1,2,\dots, n-1,m\}$ where $n\leq m\leq 2n$. If $\M$ is matched to $\N$ then $n$ must be mapped to $m$. In this case, since $m$ is taken then $(n-1)+k\in [2n]$ for any $k\in [n-1]$, contradicting the definition of matchable bases in the matroid sense. It is worth mentioning that in our example here, $\mathbb{Z}$ can be replaced with $\mathbb{Z}/p\mathbb{Z}$ for some large enough prime $p$.

\end{example}


 \section{Proofs on symmetric matchings}\label{Symmetric}
We recall that all matroids in this paper are assumed to be loopless.
\begin{proof}[Proof of Proposition \ref{prop:onlyif}(1)]
Assume  that  $0\in E(M)$. Since $M$ is  loopless,  $\{0\}$ is an independent set in $M$. Use the augmentation property  to extend $\{0\}$ to a basis $\mathcal{M}=\{0,a_2,\ldots,a_n\}$ of $M$. Let $\mathcal{N}=\{b_1,b_2,\ldots,b_n\}$ be an arbitrary basis of $N$. Then $0+ b_i=b_i\in E(M)$, for all $1\leq i \leq n$. Therefore  $\mathcal{M}$ cannot be matched to $\mathcal{N}$. This shows that $M$ is not matched to itself.
\end{proof}
\begin{remark}\label{loop remark}
    Note that eliminating the condition of not having loops may lead to the failure of certain relatively straightforward matroid analogs of previous existing results. For example, in the group setting, it is proved in \cite{Losonczy} that if $A\subset G$ is matched to itself then $0\notin A$. The matroidal counterpart of this result is presented as part (1) of Proposition \ref{prop:onlyif}. It is evident that the proof is grounded in the observation that all singletons, including $\{0\}$, are independent. In particular, we extend $\{0\}$ to form a basis for $M$. Moreover, we can construct an example demonstrating part (1) of Proposition \ref{prop:onlyif} fails if we assume having loops. That is, of a matroid with a loop that is matched with itself, and $0\in E(M)$. Let $G-\mathbb{Z}_5$, and define a matroid $M=(E, \I)$ of rank $2$ over $G$ with $E=\{\overline{0},\overline{1},\overline{3}\}$ and $\I=\{\emptyset, \{\overline{1}\}, \{\overline{3}\}, \{\overline{1},\overline{3}\}\}$. Then:
\begin{itemize}
    \item The matroid $M$ has a loop $\{\overline{0}\}$.
    \item $M$ is matched to itself; $M$ has only one basis $\mathcal{M}=\{\overline{1},\overline{3}\}$, and $\mathcal{M}$ is matched with itself through $\overline{1}\rightarrow \overline{3}$ and $\overline{3}\rightarrow \overline{1}$.
\end{itemize}
Therefore, $M$ is matched with itself, even though $\overline{0}\in E(M)$.\\
\end{remark}

\begin{proof}[Proof Theorem \ref{sparse sym main}]
Let $M$ be a sparse paving matroid with  $0\notin E(M)$.  Then by Theorem \ref{symmetric matching}, $E(M)$ is matched to itself in the group sense. That is, there exists a matching $f:E(M)\to E(M)$ such that $a+ f(a) \notin E(M)$ for every $a\in E(M)$. Let $\mathcal{M}=\{a_1,\ldots,a_n\}$ be a basis for $M$. Set $b_i:=f(a_i)$, $1\leq i\leq n$. It follows from Theorem \ref{sparse paving alternative} that the $n$-subset $\mathcal{N}=\{b_1,\ldots,b_n\}$ of $E(M)$ is either a basis for $N$ or a circuit-hyperplane. 
Note that if $\mathcal{N}$ is a basis of $M$, then $\mathcal{M}$ is matched to $\mathcal{N}$ in the matroid sense, and we are done.

So assume that  $\mathcal{N}$ is a circuit-hyperplane. Observe that $\mathcal{N}\neq E(M)$, as otherwise $r(\mathcal{N})=n$, contradicting $\mathcal{N}$ being a hyperplane. 
We claim that there exist  $x\in E(M)\setminus\mathcal{N}$ and $a_i\in \mathcal{M}$ satisfying $a_i+x\notin E(M)$. Indeed, otherwise, we have 
$
\mathcal{M}+(E(M)\setminus\mathcal{N})\subset E(M).
$
Since $\mathcal{M}$, $E(M)\setminus\mathcal{N}$ and $\mathcal{M}+(E(M)\setminus\mathcal{N})$ are all non-empty and contained in $E(M)$, and $0\notin E(M)$, by Lemma \ref{Kemperman's consequences} we have
$$
|E(M)|\geq |\mathcal{M}|+|E(M)\setminus\mathcal{N} |+1
= n+ |E(M)|-n+1=|E(M)|+1,
$$
a contradiction.

Let $\mathcal{P}=\{b_1,\ldots,b_{i-1},x,b_{i+1},\ldots,b_n\} \subset E(M)$. Since $|\mathcal{P}\cap \mathcal{N}|=n-1$ and $\mathcal{N}$ is a hyperplane, it follows from Theorem \ref{(n-1) partitions} that $\mathcal{P}$ is a basis of $M$. Now $\mathcal{M}$ is matched to $\mathcal{P}$ in the matroid sense through the map $a_i\mapsto x$ and $a_j\mapsto b_j$ for $j\neq i$. Thus $M$ is matched to itself.
\end{proof}
 \section{Proofs on asymmetric matchings} \label{Asymmetric}
 
\begin{proof}[Proof Proposition \ref{prop:onlyif}(2)]
Let $G$ be a group satisfying the matroid matching property and assume for contradiction that $G$ is neither torsion-free nor cyclic of prime order. Then there exists $a \in G$ with $o(a)=n$, where $1<n<|G|$. Consider the subgroup $H= \langle a \rangle$ of $G$. Let $M$ be the free matroid with ground set $E(M)=H$.
Choose $x\in G\setminus H$.
Let $N$ be the free matroid with ground sets $E(N)=(H\setminus \{0\})\cup \{x\}$.  
Then the only bases of  $M$ and $N$ are  $\mathcal{M}=H$ and $\mathcal{N}=\{x,a,2a,...,(n-1)a\}$, respectively. If $\mathcal{M}$ is matched to $\mathcal{N}$ in the matroid sense, then  $0$ must be mapped $x$. But then $a\mapsto ia$ for some $i\in [n-1]$, and  $a+ia \in E(M)$ for all $i$. Thus $\mathcal{M}$ cannot be matched to $\mathcal{N}$, a contradiction.
\end{proof}

\subsection{Sparse paving matroids}
In this section, we prove our main theorems concerning asymmetric matchings in sparse paving matroids. We start with the proof of 
 Theorem \ref{Asy sparse paving}, for which we will need the following Lemma from \cite{Aliabadi 1}.
\begin{lemma}\cite{Aliabadi 1}\label{matchable sets, general}
Let $G$ be an abelian group and $A$ be a finite subset of $G$ that does not contain a coset of a proper nontrivial finite subgroup of $G$. Then for any finite subset $B\subset G$ with $|A|=|B|$ and $0\notin B$, there is a matching in the group sense from $A$ to $B$.
\end{lemma}

The lemma entails:
\begin{corollary}\label{matchable sets, p(G)}
Let $G$ be an abelian group and $A$ and $B$ be finite subsets of $G$ with $|A|=|B|<p(G)$, and $0\notin B$. Then there is a matching in the group sense from $A$ to $B$. 
\end{corollary}

\begin{proof} [Proof of Theorem\ref{Asy sparse paving}] (1) 
Assume that condition (1) of the theorem holds. That is, $|E(M)|<\min\{|E(N)|-1, p(G)\}$. 
Split  into two cases:
\medskip

{\bf Case 1}.
If $M$ is a free matroid, then it has only one basis $\mathcal{M}=E(M)=\{a_1,...,a_n\}$. Let $\mathcal{N}=\{b_1,...,b_n\}$ be a basis for $N$. Since $|\mathcal{N}|=|\mathcal{M}|<p(G)$  
it follows from Corollary \ref{matchable sets, p(G)} that there is a matching $f:\mathcal{M}\to \mathcal{N}$ in the group sense. Thus $a_i+f(a_i) \notin \mathcal{M} =E(M)$ for every $i\in [n]$, implying that $M$ is matched to $N$ in the matroid sense. 
\medskip

{\bf Case 2}. 
Suppose $M$ is not a free matroid. Let $A$ be a subset of $E(N)$ with $|A|=|E(M)|$. Since $|A|=|E(M)|<p(G)$  by Corollary \ref{matchable sets, p(G)}, there exists a matching $f:E(M)\to A$ in the group sense. Let $\mathcal{M}=\{a_1,\ldots,a_n\}$ be a basis of $M$. Set $b_i:=f(a_i)$, $1\leq i\leq n$, and let  $\mathcal{N}=\{b_1,\ldots,b_n\}\subset E(N)$.
Since $N$ is a sparse paving matroid, it follows from Theorem \ref{sparse paving alternative} that  $\mathcal{N}$ is either a basis or a circuit-hyperplane of $\N$. 
If $\mathcal{N}$ is a basis of $N$, then $\mathcal{M}$ is matched to $\mathcal{N}$ in the matroid sense, and we are done. So assume  $\mathcal{N}$ is a circuit-hyperplane. 

\begin{claim}\label{cond1claim}
There exists $x\in E(N)\setminus\mathcal{N}$ and $i\in [n]$ such that
$
a_i+x\notin E(M).
$
\end{claim}
\begin{proof}
Otherwise, we have
\begin{equation}\label{eq13}
\mathcal{M}+(E(N)\setminus\mathcal{N})\subset E(M).
\end{equation}
Applying Theorem \ref{Kneser} for $\mathcal{M}$ and $E(N)\setminus\mathcal{N}$, there exists a subgroup $H$ of $G$ such that
\begin{align}
\label{eq14}
&|\mathcal{M}+(E(N)\setminus\mathcal{N})|\geq|\mathcal{M}|+|E(N)\setminus\mathcal{N}|-|H|,\\
\label{eq15}
&\mathcal{M}+(E(N)\setminus\mathcal{N}) +H= \mathcal{M}+(E(N)\setminus\mathcal{N}).
\end{align}
Note that (\ref{eq13}) and (\ref{eq15})  imply together that
$
\mathcal{M}+(E(N)\setminus\mathcal{N}) +H\subset E(M),
$
entailing
$
|H|\leq |E(M)|<p(G).
$
 Thus $H=\{0\}$.
Rewriting \eqref{eq14}, we obtain 
\begin{equation}\label{eq16}
\begin{split}
|\mathcal{M}+(E(N)\setminus\mathcal{N})| &\geq |\mathcal{M}|+|E(N)\setminus\mathcal{N}|-1 \\&=  n+|E(N)|-n-1=|E(N)|-1.
\end{split}
\end{equation}
Combining \eqref{eq16} and \eqref{eq13} we conclude that  $
|E(M)|\geq|E(N)|-1,
$
contradicting our assumption.

\end{proof}

Choose $a_i\in \mathcal{M}$ and $x\in E(N)\setminus \N$ as in the claim. Consider the set   $\mathcal{P}=\{b_1,\ldots,b_{i-1},x,b_{i+1},\ldots,b_n\}\subset E(N)$. Then $|\mathcal{P}\cap\mathcal{N}|=n-1$. Since $\mathcal{N}$ is a circuit-hyperplane, it follows from Theorem \ref{(n-1) partitions} that $\mathcal{P}$ is a basis of $N$, and $\mathcal{M}$ is matched to $\mathcal{P}$ through the map $a_i\mapsto x$ and $a_j\mapsto b_j$ for $j\neq i$. This completes the proof under the assumption of condition (1).\\
\medskip
(2)
Assume that condition (2) of the theorem holds. That is, $E(M)$ is not a progression, $|E(M)|=|E(N)|-1$, $|E(N)|<p(G)$ and $G$ is finite. We assume that $M$ is not a free matroid (an analogous argument to Case 1 of the previous part may settle the case $M$ is free.) We also assume that $n>1$ (the case n=1 is immediate as all matroids in this paper are assumed to be loopless.) Let $A$ be a subset of $E(N)$ with $|A|=|E(M)|$. Since $|A|=|E(M)|<p(G)$  by Corollary \ref{matchable sets, p(G)}, there exists a matching $f:E(M)\to A$ in the group sense. Let $\mathcal{M}=\{a_1,\ldots,a_n\}$ be a basis of $M$. Set $b_i:=f(a_i)$, $1\leq i\leq n$, and let  $\mathcal{N}=\{b_1,\ldots,b_n\}\subset E(N)$.
Since $N$ is a sparse paving matroid, it follows from Theorem \ref{sparse paving alternative} that  $\mathcal{N}$ is either a basis or a circuit-hyperplane of $N$. 
If $\mathcal{N}$ is a basis of $N$, then $\mathcal{M}$ is matched to $\mathcal{N}$ in the matroid sense, and we are done. So assume  $\mathcal{N}$ is a circuit-hyperplane. 

\begin{claim}\label{cond1claim}
There exists $x\in E(N)\setminus\mathcal{N}$ and $i\in [n]$ such that
$
a_i+x\notin E(M).
$
\end{claim}
\begin{proof}
Otherwise we have 
\begin{equation}\label{eq18}
\mathcal{M}+(E(N)\setminus\mathcal{N})\subset E(M).
\end{equation}
In a similar manner as part (1) one can verify that 
\begin{equation}
|\mathcal{M}+(E(N)\setminus\mathcal{N})| \geq |E(N)|-1,
\end{equation}
entailing
\begin{equation}
|\mathcal{M}+(E(N)\setminus\mathcal{N})| \geq |E(M)|.
\end{equation}
This together with (\ref{eq18}) implies
\begin{equation}\label{eq19}
\mathcal{M}+(E(N)\setminus\mathcal{N})=E(M).
\end{equation}
It follows from (\ref{eq19}) that $|\mathcal{M}+(E(N)\setminus\mathcal{N})|=|\mathcal{M}|+|E(N)\setminus\mathcal{N}|-1$. Therefore $\{\mathcal{M},E(N)\setminus\mathcal{N}\} $ is a critical pair. Invoking Lemma \ref{Kemperman's critical theorem} we may conclude that $\mathcal{M}$ and $E(N)\setminus\mathcal{N}$ are progressions with the same difference. This along with (\ref{eq19}), and the fact that the sumset of two progressions with the same difference is a progression, we may conclude that $E(M)$ is a progression, which is a contradiction.
\end{proof}
Choose $i\in [n]$ and $x\in E(N)\setminus \N$ as in the claim. Consider the set   $\mathcal{P}=\{b_1,\ldots,b_{i-1},x,b_{i+1},\ldots,b_n\}\subset E(N)$. In a similar manner as in the previous case, one can show that $\mathcal{P}$ is a basis of $N$ and $\mathcal{M}$ is matched to $\mathcal{P}$, in the matroid sense.

\medskip
\begin{remark}
    Note that in Claim \ref{cond1claim} since it is assumed $M$ is not a free matroid then $|E(M)|\geq n+1$, implying $|E(N)|\geq n+2$. Thus $|E(N)\setminus \N|>1$, entailing $E(N)\setminus \N$ satisfies the assumption of Lemma \ref{Kemperman's critical theorem}.
\end{remark}
\medskip

(3)
Assume that condition (3) of the theorem holds. That is, $E(M)$ is neither a progression nor a semi-progression, $G$ is finite and $|E(M)|=|E(N)|< p(G)$. Since $|E(M)|=|E(N)|<p(G)$  by Corollary \ref{matchable sets, p(G)}, there exists a matching $f:E(M)\to E(N)$ in the group sense. Let $\mathcal{M}=\{a_1,\ldots,a_n\}$ be a basis of $M$. Set $b_i:=f(a_i)$, $1\leq i\leq n$, and let  $\mathcal{N}=\{b_1,\ldots,b_n\}\subset E(N)$.
Since $N$ is a sparse paving matroid, it follows from Theorem \ref{sparse paving alternative} that  $\mathcal{N}$ is either a basis or a circuit-hyperplane of $\N$. 
If $\mathcal{N}$ is a basis of $N$, then $\mathcal{M}$ is matched to $\mathcal{N}$ in the matroid sense, and we are done. So assume  $\mathcal{N}$ is a circuit-hyperplane. Note that in this case, we must have $n>1$ as all matroids in this paper are assumed to be loopless.

\begin{claim}\label{cond3claim}
There exists $x\in E(N)\setminus\mathcal{N}$ and $i\in [n]$ such that
$
a_i+x\notin E(M).
$
\end{claim}
\begin{proof}
Assume to the contrary such $x$ and $i$ do not exist. Then we have 
\end{proof}
\begin{equation}\label{eq20}
\mathcal{M}+(E(N)\setminus\mathcal{N})\subset E(M).
\end{equation}
In a similar manner to the previous cases, one may argue that 
\begin{equation}\label{eq21}
|\mathcal{M}+(E(N)\setminus\mathcal{N})|\geq |\mathcal{M}|+|E(N)\setminus\mathcal{N}|-1.
\end{equation}
It follows from (\ref{eq20}) and (\ref{eq21}) that 
\begin{equation}\label{eq22}
|E(M)|\geq |\mathcal{M}+(E(N)\setminus\mathcal{N})|\geq |E(N)|-1=|E(M)|-1.
\end{equation}

It follows from (\ref{eq22}) that $E(M)=\mathcal{M}+(E(N)\setminus\mathcal{N})$ or $E(M)\setminus\{y\}=\mathcal{M}+(E(N)\setminus\mathcal{N})$, for some $y\in E(M)$. We split into two cases:
\medskip

{\bf Case 1}.
$E(M)=\mathcal{M}+(E(N)\setminus\mathcal{N})$. So we have
$$
E(M)=(\mathcal{M}+(E(N)\setminus\mathcal{N}))\cup \mathcal{M}=
\mathcal{M}+((E(N)\setminus \mathcal{N})\cup\{0\}).
$$
This together with $0\notin E(N)$ implies that $\{\mathcal{M},(E(N)\setminus \mathcal{N})\cup\{0\}\} $ is a critical pair. 
Applying Lemma \ref{Kemperman's critical theorem} to $\mathcal{M}$ and $(E(N)\setminus\mathcal{N})\cup\{0\}$ we may conclude that $\mathcal{M}$ and $(E(N)\setminus\mathcal{N})\cup\{0\}$ are both progressions with the same difference. Therefore, $E(M)=\mathcal{M}+
((E(N)\setminus\mathcal{N})\cup\{0\})$ is a progression, which is a contradiction.

\medskip

{\bf Case 2}.
$E(M)\setminus\{y\}=\mathcal{M}+(E(N)\setminus\mathcal{N})$. Then $\{\mathcal{M}, E(N)\setminus\mathcal{N}\}$ is a critical pair. Applying Lemma \ref{Kemperman's critical theorem} to  $\mathcal{M}$ and  $E(N)\setminus\mathcal{N}$ we may conclude that  $\mathcal{M}$ and  $E(N)\setminus\mathcal{N}$ are progressions with the same difference. Therefore, $E(M)\setminus\{y\}=\mathcal{M}+(E(N)\setminus\mathcal{N})$ is a progression. This means $E(M)$ is a semi-progression, which is a contradiction. 

\medskip
Now assume that $x$ and $i$ are obtained as in Claim \ref{cond3claim}. Set 
$$
\mathcal{P}=\{b_1,\ldots,b_{i-1},x,b_{i+1},\ldots,b_n\}\subset E(N).
$$ In a similar manner as before, one can show that $\mathcal{P}$ is a basis of $N$ and $\mathcal{M}$ is matched to $\mathcal{P}$, in the matroid sense. This completes the proof.\\

(4)
Assume that condition (4) of the theorem holds. That is, $|E(M)|<|E(N)|-n-1$.
Let $\mathcal{M}=\{a_1,...,a_n\}$ be a basis for $M$.
\begin{claim}
There exist $n$ distinct elements $b_1,b_2,\dots b_n \in E(N)$, such  that $a_i+b_i \notin E(M)$ for all $i\in [n]$.
\end{claim}
\begin{proof}
We first choose $b_1\in E(N)$ so that $a_1+b_1\notin E(M)$. Such $b_1$ exists as $|E(M)|<|E(N)|$. Next, choose $b_2\in E(N)\setminus\{b_1\}$, such that $a_2+b_2\notin E(M)$. Such $b_2 \in E(N)$ exists as $|E(M)|<|E(N)|-1$. We proceed by induction on $n$. Suppose $b_1,\dots,b_{j}$ have been chosen. We show that there exists $b_{j+1}\in E(N)\setminus \{b_1,\dots,b_j\}$ satisfying $a_{j+1}+b_{j+1}\notin E(M)$. Assume to the contrary that there 
is no such $b_{j+1}$. Then we have 
$$
a_{j+1}+(E(N)\setminus \{b_1,\dots,b_j\}\subset E(M),
$$
entailing
$$
|E(N)\setminus \{b_1,\dots,b_j\}|\leq |E(M)|.
$$
This implies 
$$
|E(M)|\geq |E(N)|-j\geq |E(N)|-n,
$$
which is a contradiction.
\end{proof}
Let $\mathcal{N}=\{b_1,\ldots,b_n\}$, where $b_1,\dots,b_n$ are as in the claim. Since $N$ is a sparse paving matroid, it follows from Theorem \ref{sparse paving alternative} that $\mathcal{N}$ is either a basis of $N$ or it is a circuit-hyperplane. If $\mathcal{N}$ is a basis, then $\mathcal{M}$ is matched to $\mathcal{N}$  and we are done.
Suppose $\mathcal{N}$ is a circuit-hyperplane. 
\begin{claim}
There exist $x\in E(N)\setminus \N$ and $i\in [n]$ such that $a_i+x\notin E(M)$.
\end{claim}
\begin{proof}
Assume to the contrary, such $x$ and $a_i$ do not exist. Then
$$
\mathcal{M}+E(N)\setminus\{b_1,\dots,b_n\}\subset E(M),
$$
entailing
$$
|E(N)|-n=|E(N)\setminus\{b_1,\dots,b_n\}|\leq |E(M)|.
$$
This contradicts $|E(M)|<|E(N)|-n-1$.
\end{proof}

Choose $i\in [n]$ and $x\in E(N)\setminus \N$ as in the claim. Then the set  $\mathcal{P}=\{b_1,\dots,b_{i-1},x,b_{i+1},\dots, b_n\}$ is a basis for $N$ and $\M$ is matched to $\mathcal{P}$, as in the previous proof. This completes the proof of the theorem under condition (4).
\medskip

\end{proof}

\begin{proof}[Proof of Proposition \ref{asymmetric, uniform}.]
 Let $A$ be a subset of $E(N)$ with $|E(M)|=|A|$. Since $|E(M)|=|A|<p(G)$ and $0\notin A$, it follows from Corollary \ref{matchable sets, p(G)} that there exists a matching $f:E(M)\to A$ in the group sense. Let $\mathcal{M}=\{a_1,\ldots,a_n\}$ be a basis of $M$, set $b_i:=f(a_i)$ and let $\N=\{b_1,\dots,b_n\}$. Since $N$ is a uniform matroid, $\N$ is a basis for $N$. Thus  $\M$ is matched to $\N$ in the matroid sense.
\end{proof}

\begin{proof} [Proof of Theorem \ref{asymmetric, sparse paving}.]
Let $M$ be a matroid over $G$  and let $N$ be  a  paving matroid over $G$, both of the same rank $n$, so that $|E(M)|=|E(N)|=n+1<p(G)$. Let  $\preceq$ be a total order on $E(M)\cup E(N) \cup (E(M)+E(N))\cup \{0\}$ that is compatible with the group structure of $G$, so that $E(M),E(N)$ are both positive, and  $\max(E(M))\notin E(M)+E(N)$. 
Let $\mathcal{M}=\{a_1,\ldots,a_n\}$ be a basis for $M$ and set $x=\max(E(M))$. Split into two cases.
\medskip

{\bf Case 1}:
$x\notin \mathcal{M}$. Note that in this case $E(M)=\M \cup \{x\}$. Let $\mathcal{N}$ be a basis for $N$. Since $|\mathcal{M}|=|\mathcal{N}|<p(G)$ and $0\notin \mathcal{N}$,  by Corollary \ref{matchable sets, p(G)} $\mathcal{M}$ is matched to $\mathcal{N}$ in the group sense. That is, there exists a bijection $f:\mathcal{M}\to \mathcal{N}$ satisfying  $a_i+f(a_i)\not\in \mathcal{M}$ for all $i\in [n]$. Moreover, since $x\notin E(M)+E(N)$, we have also $a_i+f(a_i)\neq x$ for all $i\in [n]$.  Therefore $\mathcal{M}$ is matched to $\mathcal{N}$ in the matroid sense, as needed.

\medskip

{\bf Case 2}:
 $x\in\mathcal{M}$. Without loss of generality $x=a_n$. Since  $0\notin E(N)$, and $|E(M)|=|E(N)|<p(G)$,  by Corollary  \ref{matchable sets, p(G)} there exists a matching  $f:E(M)\to E(N)$ in the group sense, that is,  $a+f(a)\not\in E(M)$ for any $a\in E(M)$.  Set $b_i:=f(a_i)$ for $i\in [n]$ and let $\mathcal{N}=\{b_1,\ldots,b_n\}$. 
 
 If $\mathcal{N}$ is a basis of $N$, then $\mathcal{M}$ is matched to $\mathcal{N}$ and we are done.
Assume $\mathcal{N}$ is not a basis of $N$. Since $N$ is a paving matroid, every subset of $E(N)$ of size $n-1$ is independent.
Therefore, the set
$\mathcal{N}_1=\{b_1,\ldots,b_{n-1}\}$ is independent. By the  augmentation property, one can extend $\mathcal{N}_1$ to a basis $\mathcal{P}=\{b_1,\ldots,b_{n-1}, b\}$ of $N$ for some $b\in E(N)$.
Since $b$ is positive and $x=a_n$ is the maximum element of $E(M)$, we have $a_n+b\notin E(M)$.  Therefore $\mathcal{M}$ is matched to $\mathcal{P}$ in the matroid sense. This completes the proof of the theorem. 
\end{proof}

\begin{proof}[Proof Theorem \ref{$|E(M)|=|E(N)|=n+1$}]
Assume to the contrary that $M$ is not matched $N$. Then there exists a basis $\mathcal{M}=\{a_1,\ldots,a_n\}$ of $M$ such that $\mathcal{M}$ is not matched to any basis of $N$. Therefore, by Proposition \ref{rank criteria}, there exists a set  $J\subset [n]$ such that  
\begin{equation}\label{rank negation}
    r_N\left(\bigcap_{i\in J}\left((-a_i+E(M))\cap E(N)\right)\right)> n-|J|.
\end{equation}
Let $J\subset [n]$ be a minimum set satisfying (\ref{rank negation}). 
Let $S=\{a_i\mid i\in J\}$ and 
 $U=\bigcap_{i\in J}\left((-a_i+E(M))\cap E(N)\right)$. Note that $U$ is not empty by its definition.

We claim that $r_N(U)=n-|J|+1$. Indeed, 
   if $|J|=1$, that is $J=\{i\}$ for some $i\in [n]$, then  (\ref{rank negation}) implies
   $$r_N(U)=r_N\left((-a_i+E(M))\cap E(N)\right) =n=n-|J|+1.$$
    Otherwise,  $|J|\geq 2$ and by  the minimality of $J$, for every $j\in J$  we have
\begin{equation*}
\begin{split}
r_N(U) &\leq r_N\left(\bigcap_{i\in J\setminus \{j\}}\left((-a_i+E(M))\cap E(N)\right)\right)\\
&\leq  n-|J\setminus \{j\}|=n-|J|+1,
\end{split}
\end{equation*}
which together with  (\ref{rank negation}) implies $r_N(U)=n-|J|+1$.

Applying Theorem \ref{Kneser} to $U$ and $S$, there exists a subgroup $H$ of $G$ for which
\begin{align}\label{Kneser 1}
  &|U+S|\geq|U|+|S|-|H|,\text{ and}\\\label{Kneser 2}
&U+S+H=U+S.  
\end{align}

By the definition of $U$ we have $U+S\subset E(M)$. 
Therefore, by (\ref{Kneser 2}) we obtain
\begin{equation*}\label{Subgroup argument}
U+S+H\subset E(M),
\end{equation*}
entailing $|H|\leq |E(M)|$.
Note that $|E(M)|<p(G)$. Indeed, this is clear if $G$ is torsion-free. If $G$ is cyclic of prime order $p$ then since $0\notin E(N)$ we have $|E(M)|=|E(N)| <p = p(G)$.
Therefore $|H|<p(G)$ and we conclude that $H=\{0\}$. 
Thus by  (\ref{Kneser 1}) we have
\begin{equation*}
n+1=|E(M)|\geq|U+S|\geq|U|+|S|-1 \geq r_N(U)+|J|-1 = n.
\end{equation*} 
Therefore, either $|U+S|=n$ or $|U+S|=n+1$. 
\medskip

{\bf Case 1}. 
$|U+S|=n$. Then $E(M)\setminus \{y\}=U+S$ for some $y\in E(M)$. We first assume that $1<|J|<n$. Then the critical pair $\{U,S\}$ satisfies the conditions of Lemma \ref{Kemperman's critical theorem}. Therefore $U$ and $S$ are progressions with the same difference. Thus $E(M)\setminus \{y\}=U+S$ is a progression entailing $E(M)$ being a semi-progression, which is a contradiction.\\
Next, we assume that $|J|=1$. Without loss of generality, we assume that $J=\{1\}$. Then $U=(-a_1+E(M))\cap E(N)$. Now $n=|U|=|(-a_1+E(M))\cap E(N)|$, which is a contradiction.\\
Finally, we assume that $|J|=n$. Then $|U|=1$. Choose 
$$
x\in U=\bigcap_{i\in [n]}\left((-a_i+E(M))\cap E(N)\right).
$$
Thus, $x\in \bigcap_{i\in [n]}\left(-a_i+E(M))\right)$. It follows from Lemma \ref{progression} that $x=0$, implying $0\in E(N)$, which is a contradiction.

\medskip

{\bf Case 2}. 
$|U+S|=n+1$. Then $E(M)=U+S$ entailing 
$$E(M)=(U+S)\cup S=(U\cup\{0\})+S.$$ 
We first assume that $|J|=1$. Without loss of generality, we assume that $J=\{1\}$. Then $U=(-a_1+E(M))\cap E(N)$. Now $n=|U|=|(-a_1+E(M))\cap E(N)|$, which is a contradiction.
\\
Next we assume that $1<|J|\leq n$. Since $\{U\cup\{0\}, S\}$ is a critical pair, by applying Lemma \ref{Kemperman's critical theorem} to $U\cup\{0\}$ and $S$ we may conclude that they are both progressions with the same difference. Therefore $E(M)=(U\cup\{0\})+S$ is a progression, which is a contradiction.\\

In both cases we extract contradictions. Thus $M$ is matched to $N$.

\end{proof}
\begin{proof}[Proof of Theorem \ref{coloopless sparse paving}.]

Since $|E(M)|=|E(N)|<p(G)$ and $0\notin E(N)$ it follows from Corollary \ref{matchable sets, p(G)} that there exists a matching $f:E(M)\rightarrow E(N)$ in the group sense. Let $E(M)=\{a_1,\dots, a_{n+1}\}$. Let $\mathcal{M}$ be a basis for $M$. Without loss of generality, we may assume that $\mathcal{M}=\{a_1,\dots,a_n\}$. Set $b_i=f(a_i)$, for all $i\in [n]$. Consider that $n$-subset $\mathcal{N}=\{b_1,\dots,b_n\}$ of $E(N)$.
\begin{claim}
$\mathcal{N}$ is a basis for $N$.
\end{claim}
\begin{proof}

Suppose for the sake of
contradiction that $\mathcal{N}$ is not a basis. It follows from Theorem \ref{sparse paving alternative} that $\mathcal{N}$ is a circuit-hyperplane. Invoking Theorem \ref{counting hyperplanes}, one may conclude that $\mathcal{N}$ is the only circuit-hyperplane of $N$. Choose $x\in E(N)\setminus \mathcal{N}$ and set $\mathcal{N}_i=(\mathcal{N}\setminus \{b_i\})\cup \{x\}$, where $i\in [n]$. Thus, every $\mathcal{N}_i$ is a basis for $N$. Hence $x$ is a coloop for $N$, a contradiction.
\end{proof}

Clearly, $\mathcal{M}$ is matched to the basis $\mathcal{N}$ for $N$ in the matroid sense, as needed.

\end{proof}

 \subsection{Transversal matroids} In this section we prove our main results concerning transversal matroids. 

\begin{proof}[Proof of Theorem \ref{partition, asy 1}.]
Let $x=\max E$ and $x'=\max E'$. Note that the  conditions of the theorem imply further that
 $x\in E_n$, $x' \in E'_n$,
 $|E|=|E'|$, and
 $|E'_i|>|E'_j|$ when $1\leq i<j\leq n$.

\begin{claim}\label{claimgen}
There does not exist $a\in E$ satisfying $a+x'\in E$.
\end{claim}
\begin{proof}
Assume to the contrary that such $a$ exists. Since $0 \prec a$ and $x\preceq x'$, it follows from the fact that the order is compatible with the group structure that  $x \prec a+x'\in E$, contradicting $x=\max E$.
\end{proof}

Let $\mathcal{M}=\{a_1,\ldots,a_n\}$ be a basis for $M$ so that $a_1\prec \cdots\prec  a_n$. Then $a_i\in E_i$ for $i\in [n]$. 
\begin{claim}
There exist elements $b_1 \in E'_{i_1},\dots, b_n \in E'_{i_n}$, such  that $a_i+b_i \notin E$, and $i_j \neq i_k$ whenever   $j\neq k$. 
\end{claim}
\begin{proof}
We first choose $b_1\in E'$ so that $a_1+b_1\notin E$. Such $b_1$ exists since otherwise, in particular,  $a_1+x'\in E$, contradicting Claim \ref{claimgen}. Let $i_1\in [n]$ be the index satisfying $b_1 \in E'_{i_1}$.
 
We proceed by induction on $n$. 
Suppose $b_1,\dots,b_{j}$ have been chosen. We show that there exists $b_{j+1}\in E'\setminus(E'_{i_1}\cup \dots \cup E'_{i_{j}})$ satisfying $a_{j+1}+b_{j+1}\notin E$.

Assume to the contrary that there 
is no such $b_{j+1}$. Then we have 
\begin{equation}\label{eq5}
a_{j+1}+(E'\setminus (E'_{i_1}\cup \dots \cup E'_{i_{j}}))\subset E.
\end{equation}
By the conditions of the theorem $a_{j+1}$ is larger than all elements of  $E_1\cup \dots \cup E_j$ and $E'$ is positive, and thus (\ref{eq5}) implies
$$a_{j+1}+(E'\setminus(E'_{i_1}\cup \dots \cup E'_{i_{j}}))\subset E \setminus (E_1 \cup \dots \cup E_j),$$
entailing
\begin{equation*}\label{eq8}
|E'|-|E'_{i_1}\cup \dots \cup E'_{i_{j}}|\leq|E|-\left|E_1 \cup \dots \cup E_j\right|.
\end{equation*}
Since $|E_i|=|E'_i|$ for all $i$, it follows that
\begin{equation*}\label{eq9}
|E'_{i_1}\cup \dots \cup E'_{i_{j}}|\ge \left|E_1 \cup \dots \cup E_j\right| = \left|E'_1 \cup \dots \cup E'_j\right|.
\end{equation*}
On the other hand, by the conditions of the theorem $|E'_i|<|E'_j|$ for $i>j$, and therefore we have
$
|E'_{i_1}\cup \dots \cup E'_{i_{j}}|= \left|E'_1 \cup \dots \cup E'_j\right|,
$
entailing
$$
E'_{i_1}\cup \dots \cup E'_{i_{j}}= E'_1 \cup \dots \cup E'_j.
$$
Rewriting (\ref{eq5}), we obtain
\begin{equation}\label{eq12}
a_{j+1}+\left(E'\setminus\left(E'_{1}\cup \dots \cup E'_{{j}}\right)\right)\subset E.
\end{equation}
\noindent
Since $x'\in E'_n$, we have  $x'\in E'\setminus\left(E'_{1}\cup \dots \cup E'_{{j}}\right)$. This together with (\ref{eq12}) implies that $a_{j+1}+x'\in E$, contradicting Claim \ref{claimgen}.
\end{proof}

Let $\mathcal{N}=\{b_1,\ldots,b_n\}$, where $b_1,\dots,b_n$ are as in the claim. Then   $\N$ is a basis of $N$, and $\mathcal{M}$ is matched to $\mathcal{N}$, completing the proof of the theorem.
\end{proof}

We are now ready to prove  Theorem \ref{asymmetric, transversal matroid}.

\begin{proof}[Proof of Theorem \ref{asymmetric, transversal matroid}.]
First, observe that Conditions 1(a),1(b) of the theorem, together with the assumption $0\notin E'$, yield $0\notin E$.
Let $\mathcal{M}=\{a_1,\ldots,a_n\}$ be a basis for $M$ so that  $a_i\in E_i$ for $i\in [n]$. Set
$E^+=\{x\in E: 0\prec x\}$
and 
$E^-=\{x\in E: x\prec 0\}.$

By Theorem \ref{partition, asy 1}, there exist $b_{k+1},\ldots,b_n$ 
in ${\bigcup_{i=k+1}^{n}}E'_i$ such that
no two  $b_j$'s belong to the same set in the partition of $E'$, and
$a_i+b_i\notin E^+$ for all $k+1\leq i\leq n$.
Since  $E_k \prec a_i $ and $a_i,b_i$ are positive,  we also have $a_i+b_i\notin E^- \cup E_k$, and therefore, $a_i+b_i\notin E$. 

Similarly, by the negative counterpart of Theorem  \ref{partition, asy 1}, there exist 
 $b_1,\ldots,b_{k-1}$ in ${\bigcup_{i=1}^{k-1}}E'_i$ such that 
no two of $b_j$'s belong to the same set in the partition of $E'$, and such that 
$a_i+b_i\notin E$.

Set $b_k:=-a_k$. Then  $b_k\in E_k$, and   $a_k+b_k=0\notin E$. Thus $\M$ is matched to the basis  $\mathcal{N}=\{b_1,...,b_n\}$ of  $N$, concluding the proof of the theorem. 
\end{proof}

\section{A concluding remark}\label{concluding}
We hope that the techniques presented here have more general
applicability, especially in the direction of generalizing these statements to abstract simplicial complexes. In this case, various substitutes for Rado's theorem on the existence of transversals in simplicial complexes have the potential to be used. Such alternatives may be found in \cite{Aharoni 1, Aharoni 2, meshulam}. 

\section*{Acknowledgements} 
We are grateful to  Shmuel Friedland for reading a preliminary version of this paper and for his useful suggestions. We also thank  Richard Brualdi for drawing our attention to the various variants of the Hall's marriage theorem. Finally, we are thankful to Steven J. Miller and David J Grynkiewicz for useful discussions on Kneser's additive theorem.\\

\textbf{Data sharing:} Data sharing not applicable to this article as no datasets were generated or analysed.\\
\textbf{Conflict of interest:} To our best knowledge, no conflict of interests, whether of financial or personal nature, has influenced the work presented in this article.

\end{document}